\documentclass[11pt, a4paper]{amsart}

\usepackage{a4wide}
\usepackage{amsmath, amsfonts, amsthm, amssymb, verbatim, enumitem, mathtools}
\usepackage[colorlinks=true,linkcolor=blue,citecolor=red]{hyperref}
\usepackage{mathrsfs}

\newcommand{\gr}{\operatorname{Gr}}
\newcommand{\Gr}{\operatorname{Gr}}
\newcommand{\pol}{\operatorname{Pol}}
\newcommand{\Aut}{\operatorname{Aut}}
\newcommand{\End}{\operatorname{End}}
\newcommand{\Pol}{\operatorname{Pol}}
\newcommand{\Clo}{\operatorname{Clo}}

\newcommand{\rA}{\mathbb{A}}

\newcommand{\aF}{\mathbf{F}}
\newcommand{\aA}{\mathbf{A}}
\newcommand{\aB}{\mathbf{B}}

\newcommand{\cC}{\mathscr{C}}
\newcommand{\cD}{\mathscr{D}}

\newcommand{\set}[1]{\{#1\}}

\newcommand{\ignore}[1]{}

\newtheorem{thm}{Theorem}[section]
\newtheorem{lem}[thm]{Lemma}
\newtheorem{prop}[thm]{Proposition}
\newtheorem{cor}[thm]{Corollary}
\theoremstyle{definition}
\newtheorem{definition}[thm]{Definition}
\theoremstyle{definition}
\newtheorem{notation}[thm]{Notation}

\title{Pseudo-loop conditions}
\author{Pierre Gillibert}
	\address{Institut f\"{u}r Diskrete Mathematik und Geometrie, FG Algebra, TU Wien, Austria}    
	\email{pgillibert@yahoo.fr}
\author{Julius Jonu\v{s}as}
	\address{Institut f\"{u}r Diskrete Mathematik und Geometrie, FG Algebra, TU Wien, Austria}    
	\email{j.jonusas@gmail.com}
\author{Michael Pinsker}
	\address{Institut f\"{u}r Diskrete Mathematik und Geometrie, FG Algebra, TU Wien, Austria, and Department of Algebra, Charles University, Czech Republic}    
	\email{marula@gmx.at}
    \urladdr{http://dmg.tuwien.ac.at/pinsker/}
    \thanks{All authors have received funding from the  Austrian Science Fund (FWF) through  project No P27600. Michael Pinsker has received funding from the Czech Science Foundation (grant No 18-20123S)}

\date{\today}

\begin{document}
\begin{abstract}
	About a decade ago, it was realized that the satisfaction of a given \emph{identity} (or \emph{equation}) of the form $f(x_1,\ldots,x_n)\approx f(y_1,\ldots,y_n)$ in an algebra is equivalent to the algebra forcing a loop into any graph on which it acts and which contains a certain finite subgraph associated with the identity. Such identities have since also been called \emph{loop conditions}, and this characterisation has produced spectacular results in universal algebra, such as the satisfaction of a \emph{Siggers identity} $s(x,y,z,x)\approx s(y,x,y,z)$ in any arbitrary non-trivial finite idempotent algebra.
	
        We initiate, from this viewpoint, the systematic study of sets of identities of the form $f(x_{1,1},\ldots,x_{1,n})\approx\cdots\approx f(x_{m,1},\ldots,x_{m,n})$, which we call \emph{loop conditions of
        width $m$}. We show that their satisfaction in an algebra is equivalent to any action of the algebra on a certain type of  relation forcing a constant tuple into the relation. Proving that for each fixed width $m$ there is a weakest loop condition (i.e., one entailed by all others), we obtain a new and short proof of the recent celebrated result stating that there exists a concrete loop condition of width~3 which is entailed in any non-trivial idempotent, possibly infinite, algebra. The framework of classical (width~2) loop conditions is insufficient for such proof.

        We then consider pseudo-loop conditions of finite width, a generalisation suitable for non-idempotent algebras; they are of the form $u_1\circ f(x_{1,1},\ldots,x_{1,n})\approx\cdots\approx u_m\circ  f(x_{m,1},\ldots,x_{m,n})$, and of central importance for the structure of algebras associated with $\omega$-categorical structures. We show that for the latter, satisfaction of a pseudo-loop condition is characterised by \emph{pseudo-loops}, i.e., loops modulo the action of the automorphism group, and that a weakest pseudo-loop condition exists (for $\omega$-categorical cores). This way we obtain a new and short proof of the theorem that the satisfaction of any non-trivial identities of height~1 in such algebras implies the satisfaction of a fixed single identity.



\end{abstract}
\maketitle

\section{Introduction}

\subsection{Mal'cev conditions}\label{sec:mal}
One of the fundamental goals of universal algebra is to draw connections between \emph{identities} which are satisfied in an algebra
$\aA$ and its \emph{structure}. An \emph{identity} (or \emph{equation}) is a formal
expression of the form
\begin{equation}\label{eq:intro1}
  u(x_1,\ldots,x_n)\approx v(y_1,\ldots,y_m),
\end{equation}
where $u,v$ are abstract \emph{terms} (in the sense of most standard textbooks
on mathematical logic, e.g., \cite{Hodges}) over some functional signature, and
$x_1, \ldots, x_n, y_1, \ldots, y_m$ are not necessarily distinct variables.
Alternatively, a logician may also say that an identity is an atomic
first-order formula over a functional signature.  The identity \eqref{eq:intro1} 
is \emph{satisfied} in an algebra $\aA$ if the functional
symbols which appear in $u$ and $v$ can be assigned \emph{term operations} of
the algebra $\aA$ in such a way that the resulting equation is true for all values
of the variables $x_1,\ldots,x_n,y_1,\ldots,y_m$ in the domain of $\aA$.
Similarly, we can define satisfaction of a set of identities in $\aA$. 
For
example, the two identities (known as \emph{Mal'cev's identities}) 
\begin{equation}\label{eq:intro2}
  m(x,x,y)\approx m(y,x,x)\approx y
\end{equation}
are satisfied in any group $(A,+,-)$: one can simply assign to $m$ the term
operation $ (x,y,z)\mapsto x-y+z$.  A set of identities is \emph{non-trivial}
if it cannot be satisfied in any algebra on an at least two-element domain whose only fundamental operations are projections.  

When we wish to relate the structure of an algebra $\aA$ with the identities it satisfies, the
precise meaning of `structure' varies 
depending on the context, and might refer, for example, to the structure of the lattice of its
congruence relations, the growth rate (in $n$) of the number of
subalgebras of its finite powers $\aA^n$, the computational complexity
of \emph{Constraint Satisfaction Problems (CSPs)} for sets of relations
invariant under $\aA$,  further identities that must be satisfied by
$\aA$, or which other algebras are contained in the \emph{variety} (a class of algebras closed under subalgebras, factoring, and products) or
the \emph{pseudovariety} (like a variety, but only closed under finite products) generated by $\aA$, just to name a few.

Historically, some of the earliest results of the type described above dealt
with the shape of the lattice of congruence relations of an algebra. For
example, an algebra $\aA$ satisfies the identities \eqref{eq:intro2} if and
only if the variety generated by $\aA$ is \emph{congruence
permutable}~\cite{maltsev-cp}; similarly sets of identities can be used to
classify \emph{congruence distributive}~\cite{jonsson-cd}, \emph{congruence
modular}~\cite{gumm-cm}, and \emph{congruence meet-semidistributive}~\cite{Olsak:2018aa}
varieties. Another recent result in this direction is the characterisation of congruence distributivity by \emph{near unanimity identities} in \emph{finitely related} algebras~\cite{barto-cd}. Following the seminal work of A.~Mal'cev~\cite{maltsev-cp}, finite conjunctions of 
identities are also referred to 
as \emph{strong Mal'cev conditions} (e.g., the aforementioned Mal'cev identities); countable disjunctions of increasingly weak strong Mal'cev conditions are simply called \emph{Mal'cev conditions}~\cite{HobbyMcKenzie}. Our results concern strong Mal'cev conditions of a specific form.

In the past few years, the interest in
the study of Mal'cev conditions has gone beyond universal algebra and has been
to a large extent inspired by applications to Constraint Satisfaction Problems,
a certain type of computational problems parametrized by relational
structures.  
In fact, the celebrated CSP complexity dichotomy for finite
structures~\cite{BulatovFVConjecture,ZhukFVConjecture} draws the line between tractability and hardness via a strong Mal'cev condition, stating that the computational
complexity of the CSP of a finite relational structure $\mathbb{A}$ is in
\textbf{P} if the algebra of all functions \emph{preserving} $\mathbb{A}$ satisfies the identity
\begin{equation}\label{equation:sigg}
  s(x,y,x,z,y,z) = s(y, x, z, x, z, y)\;,\; \quad \text{or equivalently, }\; \quad 
  q(x, y, z, x) = q(y, x, y, z)\; ,
\end{equation}
and the problem is \textbf{NP-complete} otherwise. These identities are referred to as the 6-ary and 4-ary \emph{Siggers' identities}~\cite{Siggers,
KearnesMarkovicMcKenzie}. For finite \emph{idempotent} algebras (i.e., algebras whose only unary term operation is the identity function), satisfaction of either of the Siggers' identities is equivalent to \emph{non-triviality}; here we
call an algebra \emph{non-trivial} if it satisfies any non-trivial set of 
identities. 

Some of the milestone contributions of the recent development in the
study of Mal'cev conditions were results characterising non-trivial algebras, or varieties of non-trivial algebras. This includes the above criterion for non-triviality of finite idempotent algebras, and more generally, a \emph{locally finite} variety (this includes varieties generated by a single algebra) is non-trivial if and only if it satisfies either of the Siggers' identities~\cite{Siggers,
KearnesMarkovicMcKenzie}, or equivalently, \emph{weak near unanimity identities}~\cite{MarotiMcKenzie}. Moreover, a 
finite idempotent algebra is non-trivial if and only if satisfies a \emph{cyclic
identity}~\cite{Cyclic}; and an arbitrary (possibly infinite) idempotent algebra (or locally finite variety) is non-trivial if and only if it satisfies 
\emph{Ol\v{s}\'{a}k's identities}~\cite{olsak-idempotent}, that is the set of identities
\begin{equation}\label{equation:olsak}
  o(x,y,y,y,x,x) \approx o(y,x,y,x,y,x) \approx
  o(y,y,x,x,x,y)\; .
\end{equation}

In addition to the application to CSPs referred to above, Mal'cev conditions might play an important role in the quickly developing field of \emph{Promise Constraint Satisfaction Problems} (or \emph{PCSPs)}, a generalisation of CSPs parametrised by
two relational structures. It has already been shown that if a pair of
finite relational structures is not \emph{preserved} by a function satisfying
Ol\v{s}\'{a}k's identities, then the corresponding PCSP is \textbf{NP-complete}~\cite{Bulin:2018aa}.

Naturally, the characterisation of a given structural property of algebras by identities involves the quest for weakest identities satisfied by algebras with this property. For two sets of identities $S$ and $S'$, we say that $S$ is \emph{weaker} than $S'$ , or that $S'$ \emph{implies} $S$, if every algebra  which satisfies $S'$ must also
satisfy $S$; they are \emph{equivalent} if $S$ is weaker than $S'$ and vice-versa. We also use these notions for restricted classes of algebras. For example, Siggers' identities are, for finite idempotent algebras, weaker than any other  non-trivial set of identities, and Ol\v{s}\'{a}k's identities even have this property for all idempotent algebras. The two Siggers' identities are equivalent, even for arbitrary algebras, by a deep result recently announced by M.~Ol\v{s}\'{a}k~\cite{olsak-loop2}.

\subsection{Loop conditions for graphs}
Some important Mal'cev conditions, for example the above-mentioned Siggers'
identities~\eqref{equation:sigg} and any of the \emph{cyclic identities}, are given by
a single \emph{height~1} (in short h1) identity
\begin{equation}\label{equation:h1}
  f(x_1,\ldots,x_n)\approx f(y_1,\ldots,y_n);
\end{equation}
height~1 refers to the fact that $f$ is a single functional symbol, and not an arbitrary, possibly more complex, formal term over some language (as in~\eqref{eq:intro1}). It was observed in~\cite{Siggers} that the satisfaction of such
an identity in a variety can be characterised by the property that algebras of
the variety force the existence of a loop in certain invariant graphs; in the
case of locally finite varieties, certain finite graphs.  As a consequence, identities
of the form~\eqref{equation:h1} are now sometimes referred to as \emph{loop
conditions}; we, however, provide a more general notion of loop conditions in
the next section.  The aforementioned observation provides a systematic method
for proving that a given single identity holds in a given variety: it is
sufficient to show that the associated graphs indeed always contain a loop. It is worth noting that the latter  
property had in itself already proven useful earlier in the context of Constraint
Satisfaction Problems~\cite{BulatovHColoring}, since the CSP of a graph with a loop is always solvable in constant time. 
The characterisation of identities via loops was then, for example, exploited
in~\cite{KearnesMarkovicMcKenzie} to provide a criterion of non-triviality via
4-ary Siggers' terms in locally finite varieties.  Moreover, the characterisation provides a method for comparing the \emph{relative
strength} of identities, an undertaking which was started
in~\cite{olsak-loop} and culminated in the surprising equivalence of a large
number of identities (in particular, the 4-ary and 6-ary Siggers' identities)
in~\cite{olsak-loop2}. 

\subsection{Loop conditions of finite width}
On the other hand, it is known that some Mal'cev conditions are not equivalent
to a single identity, and thus escape the method described in the previous
section. This includes, in particular, \emph{Ol\v{s}\'{a}k's
identities}~\eqref{equation:olsak}, which provides a beautiful characterisation of non-trivial
idempotent varieties: while M.~Ol\v{s}\'{a}k provided in~\cite{olsak-idempotent} the
long-sought\footnote{In fact, probably most effort was, in vain of course,
directed towards a proof of the negation of this fact.} proof of the fact that
every non-trivial idempotent variety satisfies a fixed non-trivial strong Mal'cev
condition -- for example the identities~\eqref{equation:olsak}. It has
been shown that no such condition can be given by a single
identity~\cite{Kazda-Taylor}. The proof in~\cite{olsak-idempotent} therefore
had to make use of an ad hoc adaptation of the loop technique to these
identities.

In this article we initiate the systematic study of this type of situation by
investigating, in full generality, sets of h1 identities of the form 
$$
f(x_{1,1},\ldots,x_{1,n})\approx\cdots\approx f(x_{m,1},\ldots,x_{m,n})\;.
$$
We call the number $m$ of occurrences of $f$ the \emph{width} of the set, and the set itself a \emph{loop condition of width $m$}.  We are going to show that satisfaction
of a loop condition of width $m$ in a variety can be characterised by the algebras of
the variety forcing a constant tuple into certain $m$-ary relations -- a
straightforward generalisation of a loop in graphs. Similarly to the situation
for undirected graphs~\cite{olsak-loop}, we are then going to prove that for each fixed width
$m\geq 2$, there exists a weakest condition of width $m$. We utilize this fact to provide a new, short, and relatively elementary proof of the above-mentioned theorem that non-trivial idempotent varieties can be characterised by a single loop condition (of width~3).

\subsection{Smallness assumptions}

While some notable connections between identities and structure hold for all
algebras, naturally stronger statements have been proven under additional
\emph{smallness assumptions} on the algebra $\aA$. Classical assumptions of
this kind include \emph{finiteness} of the domain of $\aA$, that the domain is
\emph{finitely generated}, that the \emph{term functions} of $\aA$ are finitely
generated, or that $\aA$ generates a \emph{locally finite variety}, i.e., a variety in which any finitely generated algebra is finite.  Even the condition of \emph{idempotency}, one of the classical frameworks in universal algebra, can be viewed as a smallness condition (of the unary functions in an algebra). 

Another,
more indirect, classical condition is \emph{finite relatedness}, which arises
from the fact that the term functions of any finite algebra $\aA$ are the set $\Pol(\mathbb{A})$ of \emph{polymorphisms} of some relational structure $\mathbb{A}$ on its domain; the polymorphisms are the finitary operations \emph{preserving} $\mathbb{A}$. The algebra $\aA$ is a \emph{finitely related}
if $\mathbb{A}$ can be chosen to have only finitely many relations. Some of the most
spectacular results (e.g.,~\cite{barto-cd, barto-cm}) have been proven (and are only true) under this
assumption.

When $\aA$ has an infinite domain, it might not be term equivalent to any
polymorphism algebra, but there is still a unique  polymorphism algebra
$\Pol(\mathbb{A})$ which agrees with the term functions of $\aA$ \emph{locally} (i.e., on
every finite set). In this case, smallness or regularity assumptions on the
relational structure $\mathbb{A}$ which come from outside the field of
universal algebra appear naturally. Such conditions might be, for example,
model-theoretic (e.g., \emph{$\omega$-categoricity}, 
\emph{homogeneity}, \emph{finite boundedness}, \emph{model-completeness}, or \emph{coreness}), combinatorial
(e.g., \emph{Ramsey-type properties}), or group-theoretic (e.g., \emph{oligomorphicity} or a certain 
\emph{orbit growth} of the automorphism group of $\mathbb{A}$) -- see e.g.~\cite{wonderland,BKOPP,Topo,BP-reductsRamsey}. One recent prime source of algebras satisfying such conditions 
are infinite domain Constraint Satisfaction Problems, where the understanding of the structure of polymorphism algebras 
has led to a great number of results of classifying the computational complexity for
various classes of $\omega$-categorical relational structures (e.g., in~\cite{BMPP16,MMSNP}).

\subsection{Pseudo-loop conditions of finite width}\label{subsect:pseudo} It follows from several recent theorems in
this area~\cite{Topo,wonderland,BartoPinskerDichotomyV3} that if the polymorphism algebra of an $\omega$-categorical structure satisfies some non-trivial finite set of h1 identities
\emph{locally} (i.e., on every finite set), then it satisfies the 6-ary
\emph{pseudo-Siggers} identity
\begin{align}
u\circ s(x,y,x,z,y,z)\approx v\circ s(y,x,z,x,z,y)\; ;\label{eq:pseudosiggers}
\end{align}
here, $u$ and $v$ are unary functional symbols, and $s$ is $6$-ary. For a certain subclass of $\omega$-categorical structures, the converse implication holds
as well~\cite{BKOPP-conf,BKOPP}, and it has been conjectured that satisfaction of this identity is the delineation of tractability and hardness of a large class of infinite-domain Constraint Satisfaction
Problems ~\cite{BPP-projective-homomorphisms, Topo}. It is known that the pseudo-Siggers identity, even
in this restricted setting, does not imply any non-trivial set of h1 identities~\cite{BMOOPW}, and hence it is indeed necessary to consider
pseudo-conditions (or other non-h1 alternatives) in this context.

The satisfaction of the pseudo-Siggers identity in the polymorphism algebra of an $\omega$-categorical structure $\mathbb{A}$ can be described by the existence of a \emph{pseudo-loop}, roughly a loop modulo the automorphism group of $\mathbb{A}$, in certain graphs invariant under finite powers of the
algebra. The theorem mentioned above which
derives the identity~(\ref{eq:pseudosiggers}) has been obtained using this
characterisation~\cite{BartoPinskerDichotomyV3, Topo}. Inspired by this fact, 
we consider the pseudo-variant of the loop conditions of finite width: that is,
we study sets of identities of the form
$$
u_1\circ f(x_{1,1},\ldots,x_{1,n})\approx\cdots\approx u_m\circ  f(x_{m,1},\ldots,x_{m,n})\; ,
$$
where $u_1,\ldots,u_m$ are unary function symbols, 
and refer to them as \emph{pseudo-loop conditions}.  Similarly to the case of
graphs, we show that satisfaction of such identities is characterised by pseudo-loops in
relations of finite arity.  We then use this to prove that there is a weakest pseudo-loop condition for polymorphism algebras of $\omega$-categorical cores.

\subsection{Summary of the results and the method}

The first result of this paper is showing that  for each fixed width $m\geq 2$,
there is a weakest non-trivial loop condition.  This
approach allows us to provide a short and, compared to the original one
in~\cite{olsak-idempotent}, elementary proof of the existence of a weakest
non-trivial strong Mal'cev condition for idempotent algebras. We perform a purely syntactic composition in order to
obtain, from a Taylor term (which exists in any non-trivial idempotent algebra by the
classical result due to W.~Taylor~\cite{Taylor}), a term which satisfies some loop
condition of width $3$ (but of unknown arity). Then we simply refer to the above-mentioned existence
of a weakest condition of width $3$. The identities we obtain are not the same as the identities~\eqref{equation:olsak}  obtained in~\cite{olsak-idempotent}; our identities are also of width~3, but have four variables.

The second result is that there exists a weakest pseudo-loop condition (of width
$2$, as it turns out) for polymorphism algebras of $\omega$-categorical core structures. This approach allows us to identify
pseudo-loop conditions which are implied by any pseudo-Taylor term in such algebras. Using this, we derive that there is a specific non-trivial pseudo-loop condition, 
similar to~(\ref{eq:pseudosiggers}), satisfied in any such algebra which satisfies some non-trivial set of h1 identities. 

While this result is weaker than the one mentioned above in
Section~\ref{subsect:pseudo} (which only requires \emph{local}, rather than
global,  satisfaction of non-trivial h1 identities in order to
derive~(\ref{eq:pseudosiggers})), our proof is considerably more elementary
than the proof of the result  in~\cite{BartoPinskerDichotomyV3}: it consists
once again of simple composition and above equivalence of pseudo-loop conditions.
Moreover, our result might pave the way to a simple proof of that theorem: it
would be sufficient to show that local satisfaction of non-trivial h1
identities implies a pseudo-Taylor term (of arbitrary arity), an undertaking
which could well turn out to be not too involved (although, of course, we do
not currently dispose of such proof). We do achieve this in the case where  
the algebra satisfies Taylor identities locally, still avoiding
the most tedious   part of the proof in~\cite{BartoPinskerDichotomyV3}.

\section{Preliminaries}

Throughout this paper we use blackboard bold letters, for example $\mathbb{A}$,
to denote relational structures, and we use bold letters, such as $\mathbf{A}$,
to denote algebras. In both cases, the same letter in plain font, in the examples above $A$, denotes the domain
set of the relational structure or the algebra.

\subsection{Function clones} A \emph{function clone} $\mathscr{C}$ is a set of functions (also called \emph{operations}) of finite arity on
a fixed domain set $C$ which contains for all $1\leq i\leq n$ the function $\pi^n_i :C^n \to C$
given by $\pi^n_i(x_1,\ldots,x_n) = x_i$, called the $i$-th $n$-ary \emph{projection}, and which  is moreover 
closed under composition: that is, whenever $n,m\geq 1$, $f\in \mathscr{C}$ is $n$-ary, and $g_1, \ldots, g_n \in \mathscr{C}$ are $m$-ary, then
$f\circ(g_1, \ldots, g_n) \in \mathscr{C}$ where   $f\circ(g_1, \ldots, g_n)$ is given by 
\[
 (x_1,\ldots,x_m) \mapsto f(g_1(x_1,\ldots, x_m),
  \ldots, g_n(x_1, \ldots, x_m)).
\]
We reserve the font $\mathscr{C}$ to denote function clones, and then write $C$ for their domain. Every function clone is the set $\Clo(\aA)$ of term functions over some algebra $\aA$, and conversely such sets of term functions always form function clones. For $n \geq 1$, we write $\mathscr{C}_{n}$ for the set of all $n$-ary
function in $\mathscr{C}$. 

Denote by $\gr(\mathscr{C})$ the set of all unary
bijective functions of $\mathscr{C}$ whose inverse is also contained in
$\mathscr{C}$. Then $\gr(\mathscr{C})$ is the largest permutation group on $C$ contained in $\mathscr{C}$. We say that $\mathscr{C}$ is
\emph{oligomorphic} if $\gr(\mathscr{C})$ is an oligomorphic permutation group, i.e., its componentwise action on any finite power of $C$ has only finitely many orbits.  

 If $I$ is a set, then $\mathscr{C}$ acts on $C^I$ by
\[
  f((x^1_i)_{i\in I}, \ldots, (x^n_i)_{i\in I}) := (f(x_i^1,\ldots,
  x_i^n))_{i\in I}\; 
\]
where $n\geq 1$, $f\in \mathscr{C}_n$, and $(x_i^j)_{i\in I} \in C^I$ for all $1\leq j\leq n$. The clone $\mathscr{C}$
acting on $C^I$ can thus be thought of as a function clone with the domain $C^I$,
which we will denote by $\mathscr{C}^I$.  Similarly, we define the power $\aA^I$ of an algebra.

\subsection{Identities}

Recall that an \emph{identity} is a formal expression $u(x_1, \ldots, x_n)
\approx v(y_{ 1}, \ldots, y_{m})$ where $u$ and $v$ are abstract terms over
some functional signature $\Omega$ and $x_1, \ldots, x_n, y_1, \ldots, y_n$ are
not necessarily distinct variables. The identity is satisfied in a function
clone $\mathscr{C}$ if the function symbols in $\Omega$ can be assigned
elements from $\mathscr{C}$ so that  $u(x_1, \ldots, x_n) \approx v(y_{ 1},
\ldots, y_{m})$ becomes a true equation between elements of $\mathscr{C}$, that
is the equation holds for all choices of values in $C$ for variables $x_1,
\ldots, x_n, y_1, \ldots, y_n$. 

The identity $u(x_1, \ldots, x_n) \approx v(y_{ 1}, \ldots, y_{m})$
is of \emph{height~1}, in short h1, if $u,v$ are themselves symbols from
$\Omega$, i.e., there is no nesting of functional symbols in $u$ and $v$. A set
of identities is h1 if all of its identities are.
A set of
identities  is satisfied \emph{locally} in a function clone $\mathscr{C}$ if for
every finite $F \subseteq C$ there is an assignment of elements of
$\mathscr{C}$ to the symbols in $\Omega$ such that the identity holds for all
$x_1, \ldots, x_n, y_1, \ldots, y_n \in F$.
 
We say that an algebra $\aA$ satisfies a set of identities if the function
clone of its term functions does (which agrees with the definition in
Section~\ref{sec:mal}). Note that in this notion of satisfaction, there is no
connection between the signature of the identities and the signature of the
algebra -- the functions in the identities are existentially quantified
and range over $\Clo(\aA)$.

More generally, a class of algebras over the same signature $\tau$ satisfies a set of identities over $\Omega$ when there is an assignment from the symbols in $\Omega$ to the formal terms over $\tau$ such that in every algebra of the class this assignment yields term functions satisfying the identity. While this is a stronger requirement than all algebras of the class satisfying the identity (via possibly different terms), the two notions coincide in the case of \emph{varieties}, i.e., classes of algebras closed under products, homomorphic images, and subalgebras, thanks to the existence of a \emph{free algebra} in the class. We remark that a considerable part of the literature is formulated in terms of varieties rather than single algebras; however, we opted to formulate our results using the latter notion, since oligomorphic algebras (i.e., algebras $\aA$ for which $\Clo(\aA)$ is oligomorphic) never form a variety.

An operation $f$ on a set $C$ is \emph{idempotent} if $f(c,\ldots,c)=c$ for all $c\in C$. A function clone (or an algebra) is idempotent if all of its functions are. It was shown by
Taylor~\cite{Taylor} that an idempotent function clone $\mathscr{C}$
satisfies some non-trivial set of identities if and only if there is $n \geq 1$ such
that a set of identities of the form 
\begin{align*}
  t(x_{1,1},\ldots, x_{1,n}) &\approx t(y_{1,1}, \ldots, y_{1, n})\\
  &\ldots\\
  t(x_{n,1},\ldots, x_{n,n}) &\approx t(y_{n,1}, \ldots, y_{n,n})
\end{align*}
is satisfied in $\mathscr{C}$, where $x_{i,i} = x$, $y_{i,i} =
y$, and $x_{i, j}, y_{i, j} \in \{x, y\}$ for all $1\leq i, j\leq n$. Any identities of this form are called \emph{Taylor identities}, and any function witnessing their satisfaction is called a \emph{Taylor function} or, in case $\mathscr{C}$ is viewed as the clone of term functions of an algebra, a \emph{Taylor term}.

\subsection{Relational structures}
A first-order formula is \emph{primitive positive} if it is logically equivalent to an existentially quantified conjunction of atomic formulas. Since we are only going to consider formulas over  relational signatures, the atomic formulas are simply equalities of variables or
relational symbols applied to some variables. We use the usual convention to
write \emph{pp} as a shortcut for primitive positive. 

A relational structure $\mathbb{A}$ is
\emph{$\omega$-categorical} if there is an, up to isomorphism, unique countable model of
the first-order theory of $\mathbb{A}$. When $A$ is countable, this is equivalent to the automorphism
group $\Aut(\mathbb{A})$ of $\mathbb{A}$ being an oligomorphic group.  A
structure $\mathbb{A}$ is a \emph{model-complete core} if every endomorphism of
$\mathbb{A}$ preserves all first-order formulas over $\mathbb{A}$.

When $n,m\geq1$, $R\subseteq A^m$, and $f\colon A^n\rightarrow A$, we say that $f$ \emph{preserves} $R$ if
$r_1,\ldots, r_n \in R$ implies that $f(r_1), \ldots, f(r_n)$, calculated componentwise, is in $R$. An algebra (or a function clone) preserves $R$ if all of its operations do. The \emph{polymorphism clone} of a relational structure $\mathbb{A}$, denoted
$\pol(\mathbb{A})$, is the function clone consisting of all finitary functions on $A$ which preserve all relations of $\mathbb{A}$. 

When a relation has a definition over a structure $\rA$ via a pp formula, it is  preserved by all functions in $\Pol(\rA)$. The converse holds when $\rA$ is at most countable and  $\omega$-categorical, or equivalently, when $\Pol(\rA)$ is oligomorphic. More generally, if $I$ is a set and a relation on $A^I$ (i.e., a subset of $(A^I)^m$ for some $m\geq 1$) is pp-definable in $\rA$, then the relation is preserved by all functions in $\Pol(\rA)^I$. Again, the converse holds when $\rA$ is $\omega$-categorical and $I$ is finite. We call any structure on some power $A^I$ of $A$ all of whose relations are pp-definable in $\rA$ a \emph{pp-power} of $\rA$.

A \emph{subuniverse} of a function clone $\mathscr{C}$ (or an algebra $\aA$) is a set $S \subseteq C$ preserved by all functions in $\mathscr{C}$ (or $\aA$).



\subsection{Topology} 
A function clone $\mathscr{C}$ comes naturally equipped with a topology known
as the \emph{pointwise topology}. The basic open sets in this topology are of the form 
\[
  \{f \in \mathscr{C}_n \;|\; f(a_1^i, \ldots a_n^i) = b^i \
  \text{for all} \ 1\leq i\leq m\}
\]
where $n, m\geq 1$, $b^i\in C$, and $(a_1^i, \ldots,a_n^i)
\in C^n$ for all $1\leq i\leq m$.
Equivalently, the topology on $\mathscr{C}_n\subseteq C^{C^n}$ is the product topology of the
discrete topology on $C$ for all $n \geq 1$, and then the topology on
$\mathscr{C}$ is the disjoint union topology of all $\mathscr{C}_n$. 

A function clone $\cC$ then is (topologically) closed within the clone of all finitary functions on $C$ if and only if it is the polymorphism clone of a relational structure on $C$.

Using topology we can also give an alternative  definition of a model-complete core for $\omega$-categorical structures: an $\omega$-categorical relational structure $\mathbb{A}$ is a model complete core if and only if 
 $\Aut(\mathbb{A})$ is dense in the endomorphism monoid $\End(\mathbb{A})$ of
$\mathbb{A}$~\cite{RandomMinOps}. It can be shown, assuming $\omega$-categoricity, that $\mathbb{A}$ is a model-complete core
if and only if every orbit of $\Aut(\mathbb{A})$ acting on
$A^n$ is pp-definable in $\mathbb{A}$ for each $n \geq 1$; this is the case if and only if every such orbit is preserved by $\Pol(\mathbb{A})$. 

It thus makes sense to call a function clone $\mathscr{C}$ a \emph{core} if $\Gr(\cC)$ is dense in the unary functions of $\cC$. If $\mathscr{C}$ is a core clone, then it is routine to verify
that so is $\mathscr{C}^n$ for every $n \geq 1$.

\subsection{Pseudo-loops} 
For the purpose of this paper a \emph{graph} is a relational structure
with a single binary relation. A \emph{loop} in a graph is an element of its domain which is related to itself. More generally, a loop in a relation (of possibly higher arity) is a constant tuple in that relation.

If $\mathscr{G}$ is a permutation group acting on a set $A$, and $R\subseteq A^m$ is a relation on that set, then a \emph{pseudo-loop of $R$ with respect to $\mathscr{G}$} is a tuple $(a_1,\ldots,a_m)\in R$ such that all $a_1,\ldots,a_m$ all lie in the same orbit of $\mathscr{G}$. A pseudo-loop with respect to the trivial group consisting only of the identity function is called a \emph{loop}; in other words, a loop is a constant tuple. If $\mathscr{A}$ is a function clone on domain $A$, then $(a_1,\ldots,a_m)\in R$ is a \emph{pseudo-loop with respect to $\mathscr{A}$} if it is a pseudo-loop with respect to $\Gr(\mathscr{A})$.

For all $k,m \geq 2$, we define $\mathbb{K}_k^m$ to be the relational structure with domain $D:=\{1,\ldots,k\}$ and a single $m$-ary relation given by 
\[
\{(x_1,\ldots,x_m)\in D^m \;|\; x_i\neq x_j\text{ for some }1\leq i,j\leq m\}\; .
\]
In particular, $\mathbb{K}_k^2$ is just a graph which forms a clique of size $k$.

Let $\mathbb{A}$ and $\mathbb{B}$ be two relational structures both with a
single $m$-ary relation $R$ and $Q$ respectively. A
\emph{homomorphism} from $\mathbb{A}$ to $\mathbb{B}$, or homomorphism from $R$
to $Q$, is $f\colon A \to B$ such that whenever $(x_1,\ldots,x_m) \in R$, then 
$(f(x_1),\ldots, f(x_m)) \in Q$.  We will be interested in structures $\rA$ such that some $\mathbb{K}_n^m$ homomorphically maps into $\rA$. Note that if such a homomorphism is not injective, then $\rA$ has a loop.


\section{A weakest non-trivial idempotent identity}

In this section we provide an alternative proof to the main theorem
of~\cite{olsak-idempotent}, which states that there exists a set of non-trivial identities which any idempotent non-trivial algebra must satisfy.  We start by giving a generalisation of the definition of a \emph{loop condition} from~\cite{olsak-loop}.

\begin{definition}
A \emph{loop condition} is a set $L$ of identities
which is of the form
\[
  f(x_{1,1},\ldots, x_{1,n}) \approx f(x_{2,1},\ldots,
  x_{2, n}) \approx \cdots \approx f(x_{m,1},\ldots,
  x_{m,n})\; ,
\]
where $n\geq 1$, $m\geq 2$, each $x_{i, j}$ is a variable from some finite
set $V$, and $f$ is an $n$-ary function symbol. We call the numbers $n$ and $m$ the \emph{arity} and the \emph{width} of the loop condition, respectively. Then
\[
  R_L := \{(x_{1,i}, \ldots, x_{m, i}) \;|\; 1\leq i\leq n\} \subseteq V^m
\]
is the ($m$-ary) \emph{relation associated with $L$}; it contains $n$ tuples.

Dually, given a relation $R\subseteq V^m$, where $V$ is a finite set and $m\geq 2$, we assign to it a loop condition $L_R$
\[
  f(r_{1,1},\ldots, r_{1,n}) \approx f(r_{2,1},\ldots,
  r_{2, n}) \approx \cdots \approx f(r_{m,1},\ldots,
  r_{m,n})\; ,
\]
where $(r_{1,1},\ldots,r_{m,1}),\ldots,(r_{1,n},\ldots,r_{m,n})$ is an enumeration of the tuples in $R$. The identities depend on the enumeration, but only up to permutation of the variables of $f$. Since we are interested in satisfaction of identities in function clones, and function clones are closed under permutations of variables of their members, we may ignore this technicality, and speak of \emph{the} loop condition $L_R$ associated with $R$.
\end{definition}

Observe that a loop condition $L$ is trivial if and only if $R_L$ contains a loop, namely the tuple $(x_{1,i}, \ldots, x_{m, i})$, where $1\leq i\leq n$ is the coordinate to which the projection satisfying $L$ projects. The next result exhibits the
relation between loop conditions and loops in greater generality.

\begin{prop}[cf. {\cite[Proposition 1]{olsak-loop}}]\label{prop:loop-lem} 
  Let $L$ be a loop condition of width $m\geq 2$, arity $n\geq 1$, and variable set $V$, let $R_L$ be the
  associated relation, and let $\aA$ be an algebra.
  The following are equivalent:
  \begin{enumerate}[label = \emph{(\alph*)}]
      \item $\aA$ satisfies $L$;

      \item for every $R \subseteq (A^{A^V})^m$ preserved
      by $\mathbf{A}^{A^V}$, if there is a homomorphism from $R_L$
      to $R$, then $R$ has a loop;

      \item for every $R \subseteq (A^{A^V})^m$ preserved
      by $\mathbf{A}^{A^V}$, if there is an injective homomorphism from $R_L$
      to $R$, then $R$ has a loop.
  \end{enumerate}
\end{prop}
\begin{proof}
  The proof is almost identical to Proposition~1 in~\cite{olsak-loop}, but we provide it for the convenience of the reader. Let $L$ be the loop condition 
  \[
  f(x_{1,1},\ldots, x_{1,n}) \approx f(x_{2,1},\ldots,
  x_{2, n}) \approx \cdots \approx f(x_{m,1},\ldots,
  x_{m,n})\; .
\]

  Assume first that~(a) holds, and let $f \in\Clo(\aA)_n$ witness this fact. Let $R$ as in~(b) be given, and let $h\colon V\rightarrow A^{A^V}$ be a mapping which sends tuples in $R_L$ to tuples in $R$. Recall that the tuples in $R_L$ are precisely the tuples $x_i:=(x_{1,i}, \ldots, x_{m, i})$, where  $1\leq i\leq n$. We have that $f(h(x_1),\ldots,h(x_n))\in R$ since $f$ preserves $R$; on the other hand, the tuple $f(h(x_1),\ldots,h(x_n))$ is constant since $f$ satisfies $L$, proving~(b). 
  
  The implication from~(b) to~(c) is trivial.
  
  Finally, suppose that~(c) holds. Let $h\colon V\rightarrow A^{A^V}$ be the
mapping which sends every variable $v$ in $V$ to the projection $\pi_v\in A^{A^V}$ defined as follows: with $A^{A^V}$ being the set of all functions from $A^V$ to $A$, and $A^V$ being the set of all tuples of elements of $A$ indexed by the elements of $V$, the function $\pi_v$ maps any tuple in $A^V$ to the element
in $A$ labelled by $v$ in the tuple. Let $R'$ be the image of $R_L$ under $h$;
we can thus write $R'=\{r_1,\ldots, r_n\}$, where
$r_j=(h(x_{1,j}),\ldots,h(x_{m,j}))$ for all $1\leq j\leq n$. Set
$$
R:=\{f(r_1,\ldots,r_{n})\;|\; f\in \Clo(\aA)_n \}\; .
$$  
Then $R$ is preserved by $\Clo(\aA)$, and $R_L$ has an injective homomorphism into $R'\subseteq R$. Hence, by~(c) the relation $R$ has a loop, which is by definition of the form $f(r_1,\ldots,r_{n})$ for some $f\in \Clo(\aA)_n$. This means that the loop is of the form
\[
f((h(x_{1,1}),\ldots,h(x_{m,1})),\ldots, (h(x_{1,n}),\ldots,h(x_{m,n})))\; ,
\]
so 
\[
f(h(x_{1,1}),\ldots,h(x_{1,n}))=\cdots= f(h(x_{m,1}),\ldots,h(x_{m,n}))\; .
\]
By the definition of $h$, this yields
\[
f(\pi_{x_{1,1}},\ldots,\pi_{x_{1,n}})=\cdots= f(\pi_{x_{m,1}},\ldots,\pi_{x_{m,n}})\; .
\]
This means that $f$ satisfies $L$.
\end{proof}


\begin{cor}\label{cor:homo} 
Let $L, L'$ be loop conditions of the same width $m\geq 2$. Suppose that there exists a homomorphism from $R_L$ to $R_{L'}$. Then $L$ implies $L'$.
\end{cor}
\begin{proof}
Item~(b) in Proposition~\ref{prop:loop-lem} for $L$ clearly implies the same statement for $L'$ by composing homomorphisms.
\end{proof}

We remark that the truth of Corollary~\ref{cor:homo} can also be seen by identifying suitable variables in any function of an algebra satisfying $L$.

\begin{notation}
For $k, m\geq 2$, let $L^m_k$ be the loop condition 
associated with the relation $\mathbb{K}^m_k$.
\end{notation}
Note that $L^m_k$ has width $m$, arity $k^m-k$, and $k$ variables.
Observe also that Ol\v{s}\'{a}k's identities
\[
  o(x,y,y,y,x,x) \approx o(y,x,y,x,y,x) \approx
  o(y,y,x,x,x,y)
\]
from the introduction are a $L^3_2$ loop condition.

In our proof of a weakest loop condition for idempotent algebras we will first derive $L^3_\ell$ for some $\ell\geq 1$ in any non-trivial idempotent algebra, by purely syntactic composition of a Taylor term. We then show that satisfaction of $L^3_\ell$ implies satisfaction of $L^3_4$. 

For the first part, we will need the following notation.  Let $f, g$ be $n$-ary and
$m$-ary operations, respectively, on the same domain. Then $f \star g$ is the $mn$-ary function
given by 
\[
  f \star g ( x_{1,1}, \ldots, x_{1, m}, x_{2, 1}, \ldots,
  x_{n,m}) := f( g(x_{1,1}, \ldots, x_{1, m}), \ldots,
  g(x_{n, 1}, \ldots, x_{n, m})).
\]  

\begin{lem}\label{lem:big-clique}
Let $n\geq 1$, and let $\aA$ be an idempotent algebra which has an $n$-ary Taylor term. Then $\aA$ satisfies $L^3_{2n}$.
\end{lem}

\begin{proof}
Since $\mathbf{A}$ satisfies some $n$-ary
  idempotent Taylor identities, it also satisfies a set of identities
  \begin{align}
    t(x_{1,1}, \ldots, x_{1, n}) &\approx t(y_{1,1}, \ldots, y_{1,n}) \notag \\
    & \ldots \label{eq-1}\\
    t(x_{n,1}, \ldots, x_{n, n}) &\approx t(y_{n,1}, \ldots, y_{n,n}) \notag
  \end{align}
  such that 
  \begin{itemize}
  \item  $x_{i, i} \neq y_{i, i}$ for all $1\leq i\leq n$;
  \item each identity contains precisely two variables;
  \item no variable occurs in two identities.
  \end{itemize}
This is because for each of the $n$ Taylor identities satisfied in $\mathbf{A}$, we can substitute different variables for $x$ and $y$ without changing the fact that they are satisfied by its Taylor term.  Let $V$ be the set of variables occurring in these identities; then $|V|=2n$.

  Let $h := t \star t \star t$; so we can write $h=t(s_1,\ldots s_n)$, where
  $$
  s_i=t( t(z_{i,1,1}, \ldots, z_{i,1, n}), \ldots,
  t(z_{i,n, 1}, \ldots, z_{i,n, n}))\; 
  $$
  for all $1\leq i\leq n$. The variables of $h$ then are given by $(z_{i,j,k}: 1\leq i,j,k\leq n)$; we will now substitute three different sets of variables from $V$ for them.  Define, for all $1\leq i,j,k\leq n$, 
$z^1_{i, j, k} := x_{i, j}$, $z^2_{i,
  j, k} := y_{i,j}$, and $z^3_{i, j, k} := x_{j, k}$. 
  
  Since $t$ is idempotent, substituting $(z_{i,j,k}^1: 1\leq i,j,k\leq n)$ for the variables $(z_{i,j,k}: 1\leq i,j,k\leq n)$ of $h$, each $s_i$ becomes $t(x_{i,1},\ldots,x_{i,n})$, and hence
  \begin{align*}
    h(({z^1_{i,j,k}:1\leq i,j,k\leq n}))\; &\;
    \approx\; t (t(x_{1,1},\ldots,x_{1,n}), \ldots, t(x_{n,1},\ldots,x_{n,n})).
  \end{align*}
  Similarly, we have 
    \begin{align*}
    h(({z^2_{i,j,k}:1\leq i,j,k\leq n}))\; &\;
    \approx\; t (t(y_{1,1},\ldots,y_{1,n}), \ldots, t(y_{n,1},\ldots,y_{n,n})).
  \end{align*}   
  It then
  follows from~\eqref{eq-1} that
  \begin{align*}
    h(({z^1_{i,j,k}:1\leq i,j,k\leq n}))\; &\;
    \approx\; t (t(x_{1,1},\ldots,x_{1,n}), \ldots, t(x_{n,1},\ldots,x_{n,n}))\\
    &\;
    \approx\; t (t(y_{1,1},\ldots,y_{1,n}), \ldots, t(y_{n,1},\ldots,y_{n,n}))\\
    &\;
    \approx\; h(({z^2_{i,j,k}:1\leq i,j,k\leq n}))\; .
  \end{align*}
  Finally, if we substitute $(z_{i,j,k}^3: 1\leq i,j,k\leq n)$ for the variables of $h$, then each $s_i$ becomes 
  $$
  t( t(x_{1,1}, \ldots, x_{1, n}), \ldots,
  t(x_{n, 1}, \ldots, x_{n, n}))\;,
  $$
  independently of $i$.  Hence, by the idempotency of $t$, we obtain again
    \begin{align*}
    h(({z^3_{i,j,k}:1\leq i,j,k\leq n}))\; &\;
    \approx\; t (t(x_{1,1},\ldots,x_{1,n}), \ldots, t(x_{n,1},\ldots,x_{n,n}))\; .
  \end{align*}   
 Summarizing, we have that the identities
  \begin{equation}\label{eq-2}
        h(({z^1_{i,j,k}:1\leq i,j,k\leq n})) \approx     h(({z^2_{i,j,k}:1\leq i,j,k\leq n})) \approx     h(({z^3_{i,j,k}:1\leq i,j,k\leq n}))
  \end{equation}
  hold in $\mathbf{A}$. 

  Next observe that for all $1\leq i,j,k\leq  n$ either $i = j$,
  and hence $z^1_{i,j,k} = x_{i,i} \neq y_{i,i} = z^2_{i,j,k}$; or $i \neq j$ in
  which case $z^1_{i,j,k} = x_{i, j} \neq x_{j, k} = z^3_{i, j, k}$. Thus, the
  relation associated with the loop condition \eqref{eq-2} is a subset of $V^3
  \setminus \{(v,v,v) \;|\; v \in  V\}$, and hence homomorphically maps into the latter (via the identity function). Hence, $\mathbf{A}$ satisfies $L^3_{2n}$ by Corollary~\ref{cor:homo}.
\end{proof}

It follows from Corollary~\ref{cor:homo} that $L^m_k$ implies $L^m_{k+1}$, for all $m,k\geq 2$, and hence for a fixed width $m\geq 2$  these conditions could become, in theory, strictly weaker with increasing arity. We now show that this is almost never the case.

\begin{lem}\label{lem:reduce-clique}
  Let $\aA$ be an algebra, let $m \geq 2$, and let $k \geq \max(4,m + 1)$. If $\aA$
  satisfies $L^m_{k + 1}$, then it also
  satisfies $L^m_k$.
\end{lem}
\begin{proof} 
  We show this by verifying the criterion given in Proposition~\ref{prop:loop-lem}~(c). Set $\aB:=\aA^{A^k}$. Let  $R \subseteq B^m$ be such that $R$
  is preserved by $\aB$ and such that $\mathbb{K}_k^m$, the relation associated with $L^m_k$, has an injective homomorphism to $R$.  Let $c_1,\ldots, c_k \in B$ be the elements in the image of that homomorphism. By Proposition~\ref{prop:loop-lem} it is sufficient to show that $R$ has a loop.

 For every $\ell\geq 1$, let $\phi_\ell^R(z_1,\ldots,z_\ell)$ be the primitive positive formula
 $$
 \bigwedge_{1\leq i_1,\ldots, i_m\leq \ell,\; \text{ not all equal}} R(z_{i_1},\ldots,z_{i_m})\; .
 $$
If $R$ contains no loops, then $\phi_\ell^R(z_1,\ldots,z_\ell)$  asserts that $z_1,\ldots,z_\ell$ are distinct and induce $\mathbb{K}_\ell^m$ in $R$. Define $Q \subseteq (B^2)^m$ by $((a_1, b_1), \ldots , (a_m, b_m)) \in Q$ if
  and only if there exist elements $x_1, \ldots, x_{k - 1} \in B$ such
  that
  \begin{enumerate}[label=(\alph*)]

    \item  $\phi_{k+2-m}^R(x_m, \ldots, x_{k - 1}, a_m, b_m)$ holds\label{con-2};

    \item $\phi_3^R(x_i, a_i, b_i)$ holds for all $1\leq i\leq m - 1$\label{con-3};

    \item for all distinct $1\leq i, j\leq  k-1$ and all  $y_3, \ldots, y_m \in \{x_1,\ldots, x_{k-1}, a_1, b_1,
    \ldots, a_m, b_m\}$ the tuple $(x_i, x_j, y_3, \ldots, y_m)$, as well as all of its permutations, is in $R$\label{con-4}.
  \end{enumerate}
  Then $(B^2, Q)$ is a pp-power of $(B, R)$, and therefore is preserved by
  $\mathbf{B}^2$.

  Let $S = \{(c_1, c_2), \ldots, (c_{k - 1}, c_k),
  (c_k, c_1), (c_1, c_3)\}$. Then $|S|=k+1$, and since $k \geq 4$, for all
  distinct $(a, b)$ and  $(a',b')$  in $S$ we have that $\{a, b\} \neq \{a',
  b'\}$. We will next show that $\phi^Q_{k+1}((c_1,c_2),\ldots,(c_{k - 1}, c_k),
  (c_k, c_1), (c_1, c_3))$ holds. To this end, let
  $(a_1, b_1), \ldots, (a_m, b_m) \in S$ be not all the same; then $(a_i, b_i) \neq (a_m,b_m)$ for some $1\leq i\leq m-1$, and we assume without loss of generality that $i=1$. We have to show that $Q((a_1, b_1), \ldots, (a_m, b_m))$ holds.
  
   Choose $x_1 \in \{a_m,
  b_m\} \setminus \{a_1, b_1\}$. Next, pick inductively for all $2\leq j\leq m-1$ an element $x_j\in \{c_1, \ldots, c_k\} \setminus \{a_j, b_j, x_1,
  \ldots, x_{j - 1}\}$; this is possible since $k \geq m + 1$. Then $|\{a_m,
  b_m, x_1, \ldots, x_{m - 1}\}| \leq m$ as $x_1 \in \{a_m, b_m\}$, and so for
  each $m\leq j\leq k-1$, we can pick a distinct element $x_j$ from
  $\{c_1, \ldots, c_k\} \setminus \{a_m, b_m, x_1, \ldots, x_{m -
  1}\}$. It follows from the above assignment that $x_1, \ldots, x_{k-1}, a_1,
  b_1, \ldots, a_m, b_m \in \{ c_1, \ldots, c_k\}$, and so in order
  for 
  \ref{con-2} and \ref{con-3} to hold we only need to observe that
  the set of variables listed in each of the conditions contains no repetitions. But this
  follows from the fact that $a_i \neq b_i$ for all $1\leq i\leq m$ and the assignment
  specified above. Finally, \ref{con-4} holds as $x_i \neq x_j$ for all distinct $1\leq i, j\leq k-1$. Hence, $Q((a_1, b_1), \ldots, (a_m, b_m))$ holds indeed, and so we have shown that $\phi^Q_{k+1}((c_1,c_2),\ldots,(c_{k - 1}, c_k),
  (c_k, c_1), (c_1, c_3))$ holds.

  The latter implies that $\mathbb{K}_{k+1}^m$, i.e., the relation associated with the loop condition $L^m_{k +1}$,   homomorphically maps into $Q$. Since $Q$ is preserved by $\mathbf{B}^2$, and $\mathbf{B}^2$ satisfies $L^m_{k +1}$ since it is a power of $\aA$, 
  it follows from Proposition~\ref{prop:loop-lem} that $Q$ contains a loop $((a,b), \ldots, (a,b))$, where $(a,b)\in B^2$. By the definition of $Q$, this means that there exist $x_1, \ldots,
  x_{k-1} \in B$ such that
  \begin{enumerate}[label=(\alph*)]

    \item  $\phi_{k+2-m}^R(x_m, \ldots, x_{k - 1}, a, b)$ holds\label{con-2};

    \item $\phi_3^R(x_i, a, b)$ holds for all $1\leq i\leq m - 1$\label{con-3};

    \item for all distinct $1\leq i, j\leq  k-1$  and all $y_3, \ldots, y_m \in \{x_1,\ldots, x_{k-1}, a, b\}$  
    the tuple $(x_i, x_j, y_3, \ldots, y_m)$, as well as all of its permutations, is in $R$\label{con-4}.
  \end{enumerate}
It is routine to verify that under these conditions $\phi_{k+1}^R(x_1, \ldots,
  x_{k - 1}, a, b)$ holds. This in turn implies that $\mathbb{K}_{k + 1}^m$, the relation associated with the loop condition $L^m_{k +1}$, homomorphically maps into $R$. Since $\aB$ satisfies that loop condition, $R$ has a loop by
  Proposition~\ref{prop:loop-lem}.
\end{proof}

It follows from Lemma~\ref{lem:reduce-clique} that there is a weakest non-trivial loop
condition of every fixed width. 

\begin{cor}\label{cor:weakestloopcond}
Let $\aA$ be an algebra which satisfies some non-trivial loop condition of width $m\geq 2$. Then $\aA$ satisfies $L^m_{\max(4,m +1)}$.
\end{cor}
\begin{proof}
Let $L$ be a non-trivial loop condition of width $m$ which is satisfied by $\aA$, and denote its variable set by $V$. Then $R_L$ homomorphically maps into $\mathbb{K}_{\ell}^m$, for any $\ell\geq |V|$. Hence, $\aA$ satisfies $L^m_\ell$ by Proposition~\ref{prop:loop-lem}, and whence also $L^m_{\max(4,m +1)}$ by Lemma~\ref{lem:reduce-clique}.
\end{proof}

Combining this with Lemma~\ref{lem:big-clique} gives us the desired result.

\begin{thm}
  Let $\aA$ be an idempotent algebra which satisfies a non-trivial set of identities. Then $\aA$ also satisfies $L^3_4$.
\end{thm} 
\begin{proof}
  By Taylor's theorem, $\aA$ has a Taylor term of some arity $n\geq 1$. By
  Lemma~\ref{lem:big-clique}, $\aA$ then satisfies $L^3_{2n}$. By Corollary~\ref{cor:weakestloopcond},  this in turn implies that $\aA$ satisfies $L^3_4$.
\end{proof}

We have seen that for a fixed width $m\geq 2$, there is a weakest loop condition, namely $L^m_{m +1}$ (Corollary~\ref{cor:weakestloopcond}). By repeating the last equation of the loop condition, and applying Proposition~\ref{prop:loop-lem}, it is also easy to see that $L^m_k$ implies $L^{m+1}_k$ for all $k\geq 2$. We now show that the converse it not true, meaning that loop conditions generally become strictly weaker with increased width.

\begin{lem}\label{L:QNUreduire2}
Let $m\ge 2$. Let $\aF$ be the free algebra, over a countable set $V$ of generators,  with a single $(m+1)$-ary operation symbol $t$ required to  satisfy the weak near unanimity identities\footnote{The term \emph{weak near unanimity operation} is commonly used for operations which satisfy the identities given here and which are moreover idempotent; we do not require, nor desire, idempotency here.}
\[
t(x,y,\dots,y)\approx t(y,x,y,\dots,y)\approx\dots \approx t(y,\dots,y,x)\; .
\]
If $t(u_1^1,\dots,u_{m+1}^1)=\dots=t(u_1^\ell,\dots,u_{m+1}^\ell)$, where $1\leq \ell<m+1$, and $u_i^k\in F$ for all $1\le i\le m+1$ and all $1\le k\le \ell$, then there exists  $1\leq i\leq m+1$ such that $u_i^1=\dots=u_i^\ell$.
\end{lem}
\begin{proof}
By deleting repeated items of the equations, we may assume that for all  $1< k\leq \ell$ we have $(u_1^1,\dots, u_{m+1}^1)\neq (u_1^{k},\dots,u_{m+1}^k)$. Fix $2\le k\le \ell$, and consider the equality $t(u_1^1,\dots,u_{m+1}^1)=t(u_1^k,\dots,u_{m+1}^k)$. By the definition of equality in the free algebra $\aF$, and since $(u_1^1,\dots, u_{m+1}^1)\neq (u_1^{k},\dots,u_{m+1}^k)$, this equality only holds if the $u_i^j$ are substituted into one of the identities satisfied by $t$. Hence, there exist $1\le i_k,j_k\le m+1$ such that $u_{i_k}^1=u_{j_k}^k$ (substituted for $y$ in the identity) and such that $u_i^1=u_j^k$ for all $1\leq i,j\leq m+1$ with $i\neq i_k$ and $j\neq j_k$ (substituted for $x$ in the identity). 

In particular, there exists $u\in F$ such that for every $1\leq k\leq \ell$ there exists at most one index  $1\leq j \leq m+1$ such that $u^k_j\neq u$. Picking any $1\leq i\leq m+1$ which never appears as this index proves the lemma.
\end{proof}

\begin{lem}\label{lem:wnu}
Let $m\ge 2$. Then the $(m+1)$-ary weak near unanimity loop condition does not imply any non-trivial loop condition of width $m$; that is, the free algebra $\aF$ as in Lemma~\ref{L:QNUreduire2} satisfies no non-trivial loop condition of width $m$.
\end{lem}

\begin{proof}
Consider any non-trivial loop condition 
\begin{equation}\label{E:reduireconditionboucle}
f(x_{1,1},\dots,x_{1,n}) \approx f(x_{2,1},\dots,x_{2,n}) \approx \dots\approx f(x_{m,1},\dots,x_{m,n})
\end{equation}
of width $m$, and suppose it is satisfied in $\aF$. We pick a term over the symbol $t$ of smallest possible depth such that the corresponding term function in $\Clo(\aF)$ satisfies the loop condition; abusing notation, we denote both the term as well as the term function in $\Clo(\aF)$ which it induces by $f$. As the loop condition is not trivial, the term $f$ is not a variable, hence we can write $f=t(u_1,\dots,u_{m+1})$, where $u_1,\dots,u_{m+1}$ have all smaller depth than $f$. In $\aF$, we have 
\begin{align*}
t( u_1(x_{1,1},\dots,x_{1,n}),\dots,u_{m+1}(x_{1,1},\dots,x_{1,n})) &= t( u_1(x_{2,1},\dots,x_{2,n}) ,\dots,u_{n+1}(x_{2,1},\dots,x_{2,n}))\\
&\dots\\
&= t( u_1(x_{m,1},\dots,x_{m,n}) ,\dots,u_{n+1}(x_{m,1},\dots,x_{m,n}) )\; .
\end{align*}
Hence it follows from Lemma~\ref{L:QNUreduire2} that there exists $1\le i\le m+1$ such that
\[
u_i(x_{1,1},\dots,x_{1,n}) \approx \cdots\approx u_i(x_{m,1},\dots,x_{m,n})\; .
\]
Therefore, $u_i$ satisfies the loop condition \eqref{E:reduireconditionboucle}, but is of smaller depth than $f$; a contradiction.
\end{proof}
\begin{cor}
Let $m\geq 2$. Then  $L_2^{m+1}$ does not imply any non-trivial loop condition of width $m$, i.e., there exists an algebra which satisfies $L_2^{m+1}$ but does not satisfy any non-trivial loop condition of width $m$.
\end{cor}
\begin{proof}
By Proposition~\ref{prop:loop-lem}, the weak near unanimity loop condition of width $m+1$ implies $L_2^{m+1}$; on the other hand, by Lemma~\ref{lem:wnu} it does not imply any non-trivial loop condition of width $m$, and the claim follows.
\end{proof}

\section{Pseudo-loop conditions}\label{sec:pseudo}

\begin{definition}
A \emph{pseudo-loop condition} is a set $P$ of identities
which is of the form
\[
  u_1\circ f(x_{1,1},\ldots, x_{1,n}) \approx u_2\circ f(x_{2,1},\ldots,
  x_{2, n}) \approx \cdots \approx u_m\circ  f(x_{m,1},\ldots,
  x_{m,n})\; ,
\]
where $n\geq 1$, $m\geq 2$, each $x_{i, j}$ is a variable from some finite
set $V$, $f$ is an $n$-ary function symbol,  and $u_1,\ldots,
u_m$ are unary function symbols.
\end{definition}
The width $m$ and arity $n$ of a pseudo-loop condition are defined as for loop conditions, and so is the relation $R_P$ associated with it. For $m,k \geq 2$, we define $pL^m_k$ to be
the pseudo-variant of the $L^m_k$ loop condition, or in other words, the pseudo-loop condition whose associated relation is $\mathbb{K}_k^m$. 

Note that since the only unary function in the projection clone is the identity function, a loop condition is non-trivial if and only if its pseudo-variant is. We remark that the precise analogue of Proposition~\ref{prop:loop-lem} can be shown for pseudo-loop conditions in core clones and pseudo-loops in relations; the proof, which we omit since we are not going to refer to this, is very similar to the one of Proposition~\ref{prop:loop-lem}. However, we will in this section be interested in implications between pseudo-loop conditions which might not hold for arbitrary algebras, but do hold under oligomorphicity and topological closedness. We therefore formulate a variant for this restricted context where, thanks to compactness, only relations on finite powers of clones have to be considered.

\begin{prop}\label{prop:ploop-lem}
  Let $P$ be a pseudo-loop condition of width $m\geq 2$ and arity $n\geq 1$, and let $\mathscr{C}$ be a closed oligomorphic core clone. The following are equivalent:
  \begin{enumerate}[label = \emph{(\alph*)}]
      \item $\mathscr{C}$ satisfies $P$ locally;

      \item $\mathscr{C}$ satisfies $P$; 

      \item for every $N \geq 1$ and every $R \subseteq (C^N)^m$
      preserved by $\mathscr{C}^N$, if there is a homomorphism from $R_{P}$ to
      $R$, then $R$ has a pseudo-loop with respect to $\mathscr{C}$;

      \item for every $N\geq 1$ and every $R \subseteq (C^N)^m$
      preserved by $\mathscr{C}^N$, if there is an injective homomorphism from $R_{P}$ to
      $R$, then $R$ has a pseudo-loop with respect to $\mathscr{C}$. 
  \end{enumerate}
\end{prop}
\begin{proof}
  Suppose that
  \[
    u_1 \circ f(x_{1,1},\ldots, x_{1,n}) \approx u_2 \circ f(x_{2,1},\ldots,
    x_{2, n}) \approx \ldots \approx u_m \circ f(x_{m,1},\ldots, x_{m,n})
  \]
  is the pseudo-loop condition $P$, and write $V := \{x_{1,1}, \ldots, x_{m,n}\}$; so $R_P\subseteq V^m$.

  The argument that (a) implies (b) 
  is analogous to the one in~\cite[Lemma~4.2]{Topo}. 

  The proof that (b) implies (c) is basically identical with the corresponding argument that (a) implies (b) in Proposition~\ref{prop:loop-lem}, but we repeat it for the convenience of the reader.  Since $P$ is satisfied in $\mathscr{C}$,
  it is also satisfied in $\mathscr{C}^N$; we will identify the symbols $f, u_1, \ldots,
  u_m$ with the functions in $\mathscr{C}^N$
  which witness the satisfaction. Fix a mapping $h 
  \colon V \to C^N$ which is a homomorphism from $R_P$ to $R$. Setting $a_i := f(h(x_{i,1}),\ldots, h(x_{i,n})) \in C^N$ for all  $1\leq i\leq m$, we have  $u_1(a_1) =  \cdots = u_m(a_m)$. Moreover, since  $h$ is a homomorphism and since $R$ is preserved by $f$ it also follows that $(a_1,\ldots, a_m) \in
  R$. Finally, $\mathscr{C}$ being a core clone implies that
  $\mathscr{C}^N$ is also a core clone. Hence there are 
  $\alpha_1,\ldots, \alpha_m \in \gr(\mathscr{C}^N)$ such that $\alpha_1(a_1) =
  \alpha_2(a_2) = \ldots = \alpha_m(a_m)$. Therefore $(a_1, \ldots, a_m)$ is a pseudo-loop of $R$ with respect to $\cC^N$.

  It is trivial that (c) implies (d).

To see that (d) implies (a), let $F \subseteq C$ be finite. Let $h\colon V\rightarrow C^{F^V}$ be the mapping which sends every variable $v$ in $V$ to the projection
\begin{align*}
\pi_v\colon  {F^V} \to C,\;\; t&\mapsto t(v)\; 
\end{align*}
onto $v$. Let $R'$ be the image of $R_P$ under $h$; we can thus write $R'=\{r_1,\ldots, r_n\}$, where $r_j=(h(x_{1,j}),\ldots,h(x_{m,j}))$ for all $1\leq j\leq n$. Set
$$
R:=\{f(r_1,\ldots,r_{n})\;|\; f\in \mathscr{C}_n^{F^V} \}\; .
$$  
Then $R$ is preserved by $\mathscr{C}^{F^V}$, and $h$ witnesses that $R_P$ has an injective homomorphism into $R'\subseteq R$. Hence, by~(c) and since $F^V$ is finite, the relation $R$ has a pseudo-loop, which is by definition of the form $f(r_1,\ldots,r_{n})$ for some $f\in \mathscr{C}_n^{F^V}$. Expanding, this means that the pseudo-loop is of the form
\[
f((h(x_{1,1}),\ldots,h(x_{m,1})),\ldots, (h(x_{1,n}),\ldots,h(x_{m,n})))\; ,
\]
so 
\[
f(h(x_{1,1}),\ldots,h(x_{1,n})),\ldots,f(h(x_{m,1}),\ldots,h(x_{m,n}))
\]
all belong to the same orbit with respect to the action of $\Gr(\mathscr{C}^{F^V})$. Hence, 
\[
u_1(f(h(x_{1,1}),\ldots,h(x_{1,n})))=\cdots= u_m(f(h(x_{m,1}),\ldots,h(x_{m,n})))
\]
for some $u_1,\ldots,u_m\in\Gr(\mathscr{C}^{F^V})$. By the definition of $h$, this yields
\[
u_1(f(\pi_{x_{1,1}},\ldots,\pi_{x_{1,n}}))=\cdots= u_m(f(\pi_{x_{m,1}},\ldots,\pi_{x_{m,n}}))\; .
\]
But this means that $u_1,\ldots,u_m,f$, viewed as functions of $\mathscr{C}$, witness the satisfaction of $P$ on $F$. Since $F$ was arbitrary, $\mathscr{C}$  satisfies $P$ locally.
\end{proof}

It is not clear whether the precise analogue of Lemma~\ref{lem:reduce-clique} holds for pseudo-loop conditions in general. Under the additional assumptions of topological closedness and oligomorphicity, it does hold for core clones. The proof is similar to that of Lemma~\ref{lem:reduce-clique}, but uses (indirectly) induction over the number of orbits of the group action. It is also worthwhile remarking that the proof moreover fails for $m> 2$; but as we shall see later, there is a satisfactory substitute showing a stronger statement in this restricted context.

\begin{lem}\label{lem:pseudo-lcondition-recude}
  Let $\mathscr{C}$ be a closed oligomorphic core clone satisfying $pL^2_{k+1}$, where $k
  \geq 4$. Then $\mathscr{C}$ satisfies $pL^2_k$. 
\end{lem}
 \begin{proof}
To prove the lemma, by  Proposition~\ref{prop:ploop-lem} it is sufficient to demonstrate that for all $N \geq 1$ and every $R \subseteq
  (C^N)^2$ preserved by $\mathscr{C}^N$, if there is an injective homomorphism from
  $\mathbb{K}^2_k$ to $R$, then $R$ has a pseudo-loop.
  We proceed by contradiction. Suppose therefore that there  exist $N \geq 1$ and $R
  \subseteq (C^N)^2$ preserved by $\mathscr{C}^N$ such that there is an injective homomorphism from $\mathbb{K}^2_k$ to $R$, but $R$ has no pseudo-loops. Since $\mathscr{C}^N$ is oligomorphic and every subuniverse of $\mathscr{C}^N$ is
  a union of orbits of $\Gr(\mathscr{C}^N)$, it follows that $\mathscr{C}^N$ has finitely many
  subuniverses. Hence, by restricting $\mathscr{C}^N$ and $R$ to a subuniverse of $\mathscr{C}^N$ if necessary, we may
  assume that no restriction of $R$ to a subuniverse of $\mathscr{C}^N$ contains an injective homomorphism from $\mathbb{K}^2_k$ to $R$. For notational convenience, we set $D:=C^N$, and $\cD:=\cC^N$.

  In the following, the pp formula $\phi_\ell^R$ is defined as in the proof of Lemma~\ref{lem:reduce-clique}. Define $Q \subseteq (D^2)^2$ by $((a_1, b_1), (a_2,b_2)) \in Q$ if and only if
  there exist elements $x_1, \ldots, x_{k - 1} \in D$ such that
  \begin{enumerate}[label=(\alph*)]
      \item $\phi_{k-1}^R(x_1, \ldots, x_{k - 1})$ holds\label{con-2-1};
      
    \item  $\phi_{k}^R(x_2, \ldots, x_{k - 1}, a_2, b_2)$ holds\label{con-2-2};

    \item $\phi_3^R(x_1, a_1, b_1)$ holds\label{con-2-3}.
  \end{enumerate}
  Then $(D^2,Q)$ is a pp-power of $(D,R)$, and so is preserved by
  $\cD^2$. We proceed exactly as in Lemma~\ref{lem:reduce-clique}
  in the case where $m=2$, and conclude that $\mathbb{K}_{k + 1}^2$ homomorphically maps into $Q$.


Recall that $\mathbb{K}_{k+1}^2$ is the
  associated relation of $L^2_{k +1}$.  Since $Q$ is preserved by $\cD^2$, the fact that $\cD$ satisfies $L^2_{k +1}$ therefore implies that $Q$ has a pseudo-loop, by 
  Proposition~\ref{prop:ploop-lem}. Denote this pseudo-loop by $((\gamma(x_k),\gamma(x_{k+1})),(x_k,x_{k+1}))$, where $\gamma\in\Gr(\cD)$, and $(x_k,x_{k+1})\in D^2$.
  By the definition of $Q$, this means that there exist $x_1,\ldots,x_{k-1}\in D$ such that
  \begin{enumerate}[label=(\alph*)]
    \item $\phi_{k-1}^R(x_1, \ldots, x_{k - 1})$ holds\label{con-2-4};

    \item  $\phi_{k}^R(x_2, \ldots, x_{k +1})$ holds\label{con-2-5};

    \item $\phi_3^R(x_1, \gamma(x_k), \gamma(x_{k+1}))$ holds\label{con-2-6}.
  \end{enumerate}
  Let $O$ be the orbit of $x_1$ with respect to $\gr(\mathscr{D})$, and set $O^+:=\{d\in D\;|\;\exists q\in O\; (R(q,d))\}
  $. 
  It follows from~\ref{con-2-4} and~\ref{con-2-6} that $x_2, \ldots,
  x_{k+1} \in O^+$. Since $\cD$ is a core, $O$ is pp-definable by $\cD$, and so is $O^+$ since its definition from $O$ is primitive positive. Hence, $O^+$ is a subuniverse of $\cD$. By~\ref{con-2-5}, $O^+$ contains a homomorphic image of 
  $\mathbb{K}^2_k$. Our minimality assumption on the counterexample then implies $D=
  O^+$, and in particular, $x_1\in O^+$. Hence, by the definition of $O^+$, there is $\beta \in \gr(\cD)$ such that
  $(x_1, \beta(x_1)) \in R$, contradicting the assumption that $R$ has no
  pseudo-loops.
\end{proof}

\begin{cor}\label{cor:weakestploop}
  Let $\mathscr{C}$ be a closed oligomorphic core clone satisfying a
  non-trivial pseudo-loop condition of width $2$. Then $\mathscr{C}$ satisfies
  $pL^2_4$. 
\end{cor}

For $n\geq 1$, we call sets of identities of the following form \emph{pseudo-Taylor} identities of arity $n$:
\begin{align}
  u_1 \circ t(x_{1,1}, \ldots, x_{1, n}) &\approx v_1 \circ t(y_{1,1}, \ldots,
  y_{1,n}) \notag \\
  & \ldots \label{eq-3}\\
  u_n \circ t(x_{n,1}, \ldots, x_{n, n}) &\approx v_n \circ t(y_{n,1}, \ldots,
  y_{n,n}) \notag
\end{align}
where $t$ is an $n$-ary functional symbol, $u_1, v_1, \ldots, u_n, v_n$ are
unary function symbols, $x_{i,i} = x$, $y_{i,i} = y$, and $x_{i, j}, y_{i, j}
\in \{x, y\}$ for all $1\leq i, j\leq n$.

The following is an adaptation, to the oligomorphic context, of an unpublished trick due to M.~Ol\v{s}\'{a}k, which he used in order to provide a new proof of the theorem stating that every finite idempotent algebra with a Taylor term satisfies $L^2_3$.

\begin{lem}\label{lem:pseudo-loop-lemma}
  Let $\mathscr{C}$ be a closed oligomorphic core clone. Let $n\geq 1$, and suppose that  $\mathscr{C}$ satisfies a set of $n$-ary pseudo-Taylor  identities. If $R\subseteq C^2$ is a binary relation preserved by
  $\mathscr{C}$ such that $\mathbb{K}^2_{2n}$ homomorphically maps to $R$, then $R$ has a  pseudo-loop with respect to $\cC$. 
\end{lem}

\begin{proof}
  Suppose that there exists a counter-example to the statement of
  the lemma, consisting of a closed oligomorphic core clone $\mathscr{C}$ and a relation $R$ satisfying the assumptions, but without a pseudo-loop. 
  Since $\mathscr{C}$ is oligomorphic and every subuniverse of $\mathscr{C}$ is
  a union of orbits of $\Gr(\mathscr{C})$, it follows that $\mathscr{C}$ has finitely many
  subuniverses. Hence, by restricting $\mathscr{C}$ and $R$ to a subuniverse of $\mathscr{C}$ if necessary, we may
  assume that no restriction of $R$ to a subuniverse of $\mathscr{C}$ admits a homomorphism from 
  $\mathbb{K}^2_{2n}$. Let $t\in \mathscr{C}_n$ be a pseudo-Taylor term, i.e., it witnesses, together with unary functions
  $u_1,\ldots u_n, v_1, \ldots v_n \in \mathscr{C}$, the satisfaction of the identities \eqref{eq-3}. 
  Since any homomorphism from $\mathbb{K}^2_{2n}$ to $R$ which is not an embedding would yield a loop in $R$, there is a set $S:=\{a_1, \ldots,a_n,b_1,\ldots b_{n}\}\subseteq C$ inducing  $\mathbb{K}^2_{2n}$ in $R$.

 Recall that $x$ and $y$ are the variables in~\eqref{eq-3}, and define for all $1\leq i\leq n$ assignments $\phi_i,\psi_i \colon\{x,y\}\rightarrow S$ to these variables defined by $\phi_i(x)=\psi_i(y):=a_i$ and $\phi_i(y)=\psi_i(x):=b_i$.
  By~\eqref{eq-3}, and since 
  $\gr(\mathscr{C})$ is dense in the unary functions of
  $\mathscr{C}$, it follows that there exist $\alpha_1, \ldots,
  \alpha_n \in \gr(\mathscr{C})$ such that
  \begin{align}
    t(\phi_1(x_{1,1}), \ldots, \phi_1(x_{1, n})) &= \alpha_1 \circ
    t(\phi_1(y_{1,1}), \ldots, \phi_1(y_{1,n})) \notag\\
    & \ldots  \label{eq-5}\\
    t(\phi_n(x_{n,1}), \ldots, \phi_n(x_{n, n})) &= \alpha_n
    \circ t(\phi_n(y_{n,1}), \ldots, \phi_n(y_{n,n})). \notag
  \end{align}
  Moreover, any two elements among 
  \begin{align}
    t(\phi_1(x_{1,1}), &\ldots, \phi_1(x_{1, n}))\notag\\
    t(\psi_1(x_{1,1}), &\ldots, \psi_1(x_{1, n}))\notag\\ 
    & \ldots \label{eq-4}\\
    t(\phi_n(x_{n,1}), &\ldots, \phi_n(x_{n, n}))\notag\\ 
    t(\psi_n(x_{n,1}), &\ldots, \psi_n(x_{n, n}))\notag
  \end{align}
  are related in $R$. To see this, note that $R$ is preserved by $t$ and that
  every argument of $t$ which appears in this list is an element of $S$, which induces $\mathbb{K}^2_{2n}$. Hence it is sufficient to show that for every $1\leq j\leq n$ the elements $\phi_1(x_{1,j}), \psi_1(x_{1,j}), \ldots,
  \phi_n(x_{n,j}), \psi_n(x_{n,j})$ are pairwise distinct. This is true for all the pairs $\phi_i(x_{i,j}), \psi_i(x_{i,j})\in\{a_i,b_i\}$, since $\phi_i(x_{i,j})=a_i$ if and only if $\psi_i(x_{i,j})=b_i$; moreover, the range of $\phi_i$ is disjoint from that of $\phi_k$ and that of $\psi_k$ whenever $i\neq k$, and vice-versa. Since $R$ does not contain any loop, it follows that the elements of~\eqref{eq-4} are pairwise distinct and induce $\mathbb{K}^2_{2n}$ in $R$.

We next claim that $t(a_1,\ldots,a_n)$ is related in $R$ to $t(\psi_i(x_{i,1}),\ldots, \psi_i(x_{i, n}))$ for all $1\leq i\leq n$. To this end, note that whenever $j\neq i$, then $a_j$ is obviously related to $\psi_i(x_{i,j}) \in \{a_i,b_i\}$; moreover, $a_i$ is related to $\psi_i(x_{i,i})=\psi_i(x)=b_i$, hence the claim follows from the preservation of $R$ by $t$.

Similarly, $t(a_1,\ldots,a_n)$ is related in $R$ to $t(\phi_i(y_{i,1}),\ldots, \phi_i(y_{i, n}))$ for all $1\leq i\leq n$; this time, we use the fact that $\phi_i(y_{i,i})=b_i$. Hence, by~\eqref{eq-5}, we see that $\alpha_i(t(a_1,\ldots,a_n))$ is related to 
\[
\alpha_i(t(\phi_i(y_{i,1}),\ldots, \phi_i(y_{i, n})))=t(\phi_i(x_{i,1}),\ldots, \phi_i(x_{i, n}))\; .
\]

It follows that every element listed in~\eqref{eq-4} is related to an element in the orbit $O$ of $t(a_1,\ldots,a_n)$. Since $\mathscr{C}$ is a core clone, $O$ is preserved by $\cC$, and hence so is the set of its neighbours $O^+:=\{c\in C\;|\;\exists q\in O (R(q,c))\}$, since this definition is primitive positive. Since $R$ has no pseudo-loops, $O^+$ is a proper subuniverse of $\mathscr{C}$. By the above, the elements in~\eqref{eq-4} are contained in $O^+$. Since they induce $\mathbb{K}^2_{2n}$ in $R$, this contradicts our minimality assumption.
\end{proof}

Although the proof of Lemma~\ref{lem:pseudo-lcondition-recude}, which reduces the arity of pseudo-loop conditions, only works for width~$2$, the following lemma implies in particular that we can reduce the width $m$ of an arbitrary loop condition down to $2$, permitting the use of Lemma~\ref{lem:pseudo-lcondition-recude}.

\begin{lem}\label{lem:pTaylor2loopcond}
Let $\mathscr{C}$ be a closed oligomorphic core clone, and let $n\geq 1$. If $\cC$ has an $n$-ary pseudo-Taylor term, then it satisfies $pL^2_{2n}$.
\end{lem}
\begin{proof}
  Let $N \geq 1$, let
  $R \subseteq (C^N)^2$ be preserved by $\mathscr{C}^N$, and suppose
  $\mathbb{K}^2_{2n}$ homomorphically maps into $R$. Then $R$ has a pseudo-loop by 
  Lemma~\ref{lem:pseudo-loop-lemma}. The lemma thus follows from Proposition~\ref{prop:ploop-lem}.
\end{proof}

\begin{thm}\label{thm:pTaylor2L24}
  Let $\mathscr{C}$ be a closed oligomorphic core clone which has a pseudo-Taylor term.
  Then $\mathscr{C}$ satisfies $pL^2_4$.
\end{thm}
\begin{proof}
   This is a direct consequence of Corollary~\ref{cor:weakestploop} and Lemma~\ref{lem:pTaylor2loopcond}.
\end{proof}

\begin{lem}\label{lem:noh12pTaylor}
  Let $\mathscr{C}$ be a closed oligomorphic core clone which satisfies a non-trivial set of h1 identities. Then $\mathscr{C}$ has a pseudo-Taylor term.
\end{lem}
\begin{proof}
Let 
\begin{align}
  f_1(x_1,\ldots, x_n) &\approx g_1(x_1,\ldots,x_n) \notag \\
  & \ldots \label{eq-7}\\
  f_m(x_1,\ldots, x_n) &\approx g_m(x_1,\ldots,x_n) \notag
\end{align}
be a non-trivial set of h1 identities satisfied by $\mathscr{C}$. Let $F\subseteq C$ be finite. Then there exist $\alpha_1,\ldots,\alpha_m\in\Gr(\mathscr{C})$ such that the functions $\alpha_i\circ f_i$ and $\alpha_i\circ g_i$ are idempotent on $F$, for all $1\leq i \leq m$; moreover, these functions still satisfy above identities, and so we may assume that the original functions were idempotent on $F$. Set $t:=f_1\star \cdots\star f_m\star g_1\star\cdots\star g_m$, and let $\ell:=n^{2m}$ be its arity. Then for all $1\leq i \leq m$, we have that
\begin{align}
f_i(x_1,\ldots, x_n)=u^{i,f}\circ t(z_1^{i,f},\ldots,z_\ell^{i,f})\label{eq:ft}
\end{align}
for some unary function $u^{i,f}\in\cC$ and a suitable choice of variables $z_1^{i,f},\ldots,z_\ell^{i,f}\in\{x_1,\ldots, x_n\}$; a similar statement holds for $g_i$. Hence, we have that the system of identities
\begin{align*}
  u^{1,f}\circ t(z_1^{1,f},\ldots,z_\ell^{1,f}) &\approx u^{1,g}\circ t(z_1^{1,g},\ldots,z_\ell^{1,g}) \notag \\
  & \ldots \label{eq-8}\\
  u^{m,f}\circ t(z_1^{m,f},\ldots,z_\ell^{m,f}) &\approx u^{m,g}\circ t(z_1^{m,g},\ldots,z_\ell^{m,g}) \notag 
\end{align*}
holds in $\cC$. This system is non-trivial: if $t$ could be assigned an $\ell$-ary projection, and the unary symbols $u^{i,f}, v^{i,f}$ the identity, so that the identities of the system become true equations, then also the system~\eqref{eq-7} would be satisfiable by projections, by virtue of~\eqref{eq:ft}. Non-triviality means that for every $1\leq j\leq \ell$ there exists $1\leq i\leq m$ such that $z_j^{i,f}\neq z_j^{i,g}$. Pick one such $i$ for each $j$, and call it $\phi(j)$. By repeating identities (and thus enlarging $m$), this assignment $\phi\colon \{1,\ldots,\ell\}\rightarrow\{1,\ldots,m\}$ can be made injective, by deleting identities, it can be made bijective, so then $m=\ell$. The system of identities obtained by replacing, for all $1\leq j\leq \ell$, all occurrences of the variable $z_j^{\phi(i),f}$ in the $\phi(i)$-th identity by $x$, and all other variables of that identity by $y$, is still satisfied by $t$ in $\cC$, and shows that $t$ satisfies a set of pseudo-Taylor identities on $F$. Since $F$ was arbitrary, and the shape of the pseudo-Taylor identities obtained does not depend on $F$, but only on the shape of the system of h1 identities, the same pseudo-Taylor identities are satisfied on every finite set (though possibly by different terms). The proof in~\cite[Lemma~4.2]{Topo}, already invoked in Proposition~\ref{prop:ploop-lem}, then shows that these pseudo-Taylor identities are satisfied in $\cC$.
\end{proof}

\begin{thm}\label{thm:noh122L24}
  Let $\mathscr{C}$ be a closed oligomorphic core clone which satisfies a non-trivial set of h1 identities.
  Then $\mathscr{C}$ satisfies $pL^2_4$.
\end{thm}
\begin{proof}
This is the direct consequence of Lemma~\ref{lem:noh12pTaylor} and~\ref{thm:pTaylor2L24}.
\end{proof}

\section{Discussion}

\subsection{From local to global}

While we have derived $pL^2_4$ from any non-trivial set of h1 identities for closed oligomorphic core clones, it is known that the local satisfaction of such identities would already be sufficient to obtain $pL^2_3$. However, the proof of this fact uses the pseudo-loop lemma from~\cite{Topo,BartoPinskerDichotomyV3}, the proof of which contains a part (precisely, Lemma 3.5 in~\cite{Topo}) which is still considered non-satisfactory due to its ad hoc nature. The following theorem shows that in the case of local Taylor identities, which is a stronger assumption, we can derive $pL^2_3$ while avoiding that part.

\begin{thm}\label{thm:local2global}
  Let $\mathscr{C}$ be a closed oligomorphic core clone which
  satisfies Taylor identities locally, i.e., for every finite set $F\subseteq C$ there is a function in $\mathscr{C}$ satisfying some set of Taylor identities on $F$. Then $\mathscr{C}$
  satisfies $pL^2_3$.
\end{thm}
\begin{proof}
  Let $N\geq 1$, let $R
  \subseteq (C^N)^2$ be preserved by $\mathscr{C}^N$, and suppose that $R$ contains a homomorphic image of $\mathbb{K}^2_3$. By Proposition~\ref{prop:ploop-lem}, it is
  sufficient to show that $R$ has a pseudo-loop. Suppose that this is not the
  case.

  By the same argument as in the \emph{Steps 0 -- 4} of
  \cite[Lemma~3.1]{Topo}, we may assume that $R$ contains no diamonds, that is, there are no distinct
  $x_1, x_2, x_3, x_4 \in C^N$ such that $\{x_1, x_2, x_3\}$ and $\{x_2,x_3,
  x_4\}$ both induce $\mathbb{K}^2_3$. Pick $a, b, c \in C^N$ inducing $\mathbb{K}^2_3$  in $R$, and let $t$ be a local Taylor term on $\{a, b, c\}$. Denoting the arity of $t$ by $n$, we then have that $t$ is a homomorphism from $(\mathbb{K}^2_3)^n$
  to $R$. By
  \cite[Claim 3, Subsection 3.2]{BulatovHColoring}, the image $S$ of $\{a,b,c\}$ under $t$ then must induce a graph isomorphic to 
  $(\mathbb{K}^2_3)^m$ for some $1\leq m\leq n$, and so there is a graph
  homomorphism $\phi$ from this induced graph to the graph $\mathbb{K}^2_3$ induced by $\{a, b, c\}$. 

  Therefore, function $\phi \circ t   
   \colon \{a,b,c\}^n \to \{a, b, c\}$ preserves the relation $R$ (restricted to $\{a,b,c\}$). Since $t$ satisfies some Taylor identities on $\{a, b, c\}$, so does $\phi
  \circ t$. It is, however, well-known that all polymorphisms of $\mathbb{K}^2_3$ depend on only one variable; a contradiction.
\end{proof}

One elegant way to obtain $pL^2_4$ from the \emph{local} satisfaction of non-trivial sets of h1 identities for closed oligomorphic core  clones could be to show that this assumption implies a pseudo-Taylor term of some large arity, and then apply Theorem~\ref{thm:pTaylor2L24}. Another way would be to generalize the proof of Theorem~\ref{thm:local2global} to this more general situation, perhaps deriving a loop condition $pL^m_k$ of larger width or arity, and then applying Theorem~\ref{thm:pTaylor2L24}. We do, however, not dispose of such proofs, and leave this note as a suggestion for future work.

\subsection{Low arity loop conditions}

We have seen that for each fixed width $m\geq 3$, there exists a weakest non-trivial loop condition, namely $L^{m}_{m+1}$. All loop conditions $L^m_{k}$, where $k\geq {m+1}$ are equivalent to it; it would be interesting to know how the loop conditions $L^m_k$ for $2\leq k< {m+1}$ relate to $L^{m}_{m+1}$, and to each other. For $m=2$ this is well-understood~\cite{olsak-loop}.

\subsection{Pseudo-loop conditions}
We have shown that the pseudo-loop conditions of width~2 corresponding to graphs which are cliques of size at least 4 are all equivalent (in the context of closed oligomorphic core clones). In the case of loop conditions, M.~Ol\v{s}\'{a}k has proved the equivalence of all conditions corresponding to odd undirected cycles~\cite{olsak-loop}, which, together with the equivalence of conditions corresponding to cliques, then easily yields the equivalence of all loop conditions whose relation is a non-bipartite undirected graph. The case of undirected odd cycles remains open for pseudo-loop conditions, and the equivalence of pseudo-loop conditions whose relations is a non-bipartite undirected graph figures among the most interesting open problems in this direction.

\bibliographystyle{plain}
\bibliography{global.bib,identities.bib}
\end{document}